\documentclass[12pt]{article}
\usepackage[utf8]{inputenc}
\usepackage{amsthm}
\usepackage{amsmath}
\usepackage{amssymb}
\usepackage{amscd}
\usepackage[all]{xy}
\usepackage[margin=1in]{geometry}
\usepackage{comment}
\usepackage{inputenc}
\usepackage{pdfpages}
\usepackage{tikz}
\usepackage{pgfplots}
\usepackage{latexsym}
\usepackage{tikz-cd}
\usetikzlibrary{patterns}

\begin{document}

\title{On the integral domains characterized by a Bezout Property on intersections of principal ideals}

\author{Lorenzo Guerrieri \thanks{ Jagiellonian University, Instytut Matematyki, 30-348 Krak\'{o}w \textbf{Email address:} guelor@guelan.com }  \and K. Alan Loper
\thanks{ Department of Mathematics, Ohio State University – Newark, Newark, OH 43055  \textbf{Email address:}  lopera@math.ohio-state.edu }}

\maketitle

\begin{abstract}
\noindent In this article we study two classes of integral domains.  The first is characterized by having a finite intersection of principal ideals being finitely generated only when it is principal. The second class consists of the integral domains in which a finite intersection of principal ideals is always non-finitely generated except in the case of containment of one of the principal ideals in all the others. We relate these classes to many well-studied classes of integral domains, to star operations and to classical and new ring constructions.

\medskip

\noindent MSC: 13A15, 13F15, 13A18, 13F05, 13G05.\\
\noindent Keywords: intersections of principal ideals, star operations, ring constructions.
\end{abstract}

\newtheorem{theorem}{Theorem}[section]
\newtheorem{lemma}[theorem]{Lemma}
\newtheorem{prop}[theorem]{Proposition}
\newtheorem{corollary}[theorem]{Corollary}
\newtheorem{problem}[theorem]{Problem}
\newtheorem{construction}[theorem]{Construction}

\theoremstyle{definition}
\newtheorem{defi}[theorem]{Definitions}
\newtheorem{definition}[theorem]{Definition}
\newtheorem{remark}[theorem]{Remark}
\newtheorem{example}[theorem]{Example}
\newtheorem{question}[theorem]{Question}
\newtheorem{comments}[theorem]{Comments}

\newtheorem{discussion}[theorem]{Discussion}

\newcommand{\N}{\mathbb{N}}
\newcommand{\m}{\mathfrak{m}}
\newcommand{\p}{\mathfrak{p}}
\newcommand{\pp}{\mathcal{P}}
\newcommand{\Z}{\mathbb{Z}}
\newcommand{\Q}{\mathbb{Q}}
\newcommand{\F}{\mathcal{F}}
\newcommand{\h}{\mathcal{H}}
\newcommand{\G}{\mathcal{G}}
\newcommand{\I}{\mathcal{I}}
\def\min{\mbox{\rm min}}
\def\max{\mbox{\rm max}}
\def\ff{\frak}
\def\Spec{\mbox{\rm Spec }}
\def\Proj{\mbox{\rm Proj }}
\def\hgt{\mbox{\rm ht }}
\def\type{\mbox{ type}}
\def\Hom{\mbox{ Hom}}
\def\rank{\mbox{ rank}}
\def\Ext{\mbox{ Ext}}
\def\Ker{\mbox{ Ker}}
\def\Max{\mbox{\rm Max}}
\def\End{\mbox{\rm End}}
\def\xpd{\mbox{\rm xpd}}
\def\Ass{\mbox{\rm Ass}}
\def\emdim{\mbox{\rm emdim}}
\def\epd{\mbox{\rm epd}}
\def\repd{\mbox{\rm rpd}}
\def\ord{\mbox{\rm ord}}
\def\gcd{\mbox{\rm gcd}}
\def\Tr{\mbox{\rm Tr}}
\def\Res{\mbox{\rm Res}}
\def\Ap{\mbox{\rm Ap}}
\def\sdefect{\mbox{\rm sdefect}}
\maketitle

\section{Introduction}

The focus of this paper is on intersections of finite collections of principal ideals of integral domains.  This problem has been well studied in term of finiteness conditions on the intersections.  An integral domain is classically called a \it GCD domain \rm if the intersection of 
two principal ideals is always principal (this is known to be equivalent to the existence of the gcd for each pair of elements).  Of course, this condition is trivially satisfied in a PID, since then every ideal is principal, and in more  generality  for a Noetherian domain the GCD condition is equivalent to  existence of the unique factorization into irreducibles  for each element.  A domain is a \it finite conductor \rm domain.
if the intersection of two principal ideals is always finitely generated.  More generally, a domain is a \it coherent domain \rm if the intersection of two finitely generated ideals 
is still finitely generated.  Numerous variations on these ideas have been proposed and well-studied, see for some general reference \cite{Gil}, \cite{glaz2}, \cite{glaz3}, \cite{z2}.  Generally, the theme of these studies has been to put fine
variations on what types of ideals are intersected and on what well-behaved result should be expected.  The goal of this paper is somewhat different.  

In \cite{Gue}, infinite chains of local monoidal transforms of regular local rings are studied in order to determine when the directed union of such a chain has GCD-type properties.  A remarkable 
occurrence observed in this study is that sometimes, rather than a GCD domain, the process yields a domain which is, in a sense, as far as possible from being
a GCD domain.   In particular, domains arise in which the intersection of a finite collection of principal ideals is almost never finitely generated.  Our goal is to 
study such domains.  Two slightly different classes - Bezout Intersection Domains (BID) and Strong Bezout Intersection Domains (SBID) are investigated.  (These 
are defined in Section 2.)    

There are several reasons to study such domains.  They are, perhaps surprisingly, easy to construct, and have numerous elegant properties.  Also, there is a
potential connection to an old unsolved question of Vasconcelos.  As noted above, a domain is called coherent provided the intersection of two finitely generated
ideals is still finitely generated.  Noetherian domains and Pr\"ufer domains are both classes of domains that are always coherent.  Vasconcelos asked whether
the integral closure of a one-dimensional coherent domain is always a Pr\"ufer domain (see \cite[Problem 65]{Chapman} and for more detailed references see \cite[Chapter 5-7]{glaz3}).  It is easy to show that a one-dimensional, integrally closed, local domain
which is not Pr\"ufer is a SBID - and hence very far from being coherent.  Hence, a deep understanding of strong Bezout intersection domains might lead 
to progress on Vasconcelos' question by leading to a proof that such a domain cannot be the integral closure of a coherent domain.

The study of BID and SBID in higher Krull dimension could be helpful to deal with a 
more general form of Vasconcelos' question asking if the integral closure of a finite conductor domain has to be a PvMD (\cite[Remark 21]{z2}).

In this article, in Section 2 we define our classes of domains and prove some elementary properties.  Of particular note is Theorem \ref{k+m} which shows that the classical
$k + \m$ construction yields a wide class of examples of SBID and the subsequent Theorem \ref{pullback} in which the case of more general pullbacks of local domains is considered.  We also give a surprising construction (Theorem \ref{krulldomain}) of a Krull domain which is also 
BID.  

In Section 3 we define an operation on the class of ideals that we label as the $\xi$ operation.  This is a variant of the classical $w$ operation, which is,
in turn, a variant of the classical $t$ operation.  We demonstrate a lot of interplay between these three operations and our classes of domains.  
Especially noteworthy is Theorem  \ref{carattsbid} which gives several equivalent conditions for a domain to be SBID in terms of the interplay between
$\xi$ and $w$ and $t$. We also relate here BID and SBID conditions with locally cyclic ideals and Pre-Schreier domains.

In Section 4 we consider a method for constructing examples of SBID which also generalize classical
$k + \m$ constructions.  This method is reminiscent of the Rees ring construction in the fact that 
indeterminates are added to the base ring, but unlike in a standard polynomial ring, these indeterminates have coefficients that come from an ideal of the base ring. Such examples may often be not integrally closed, but
especially remarkable here is the construction (Theorem \ref{cic}) of SBID which are instead completely integrally closed.

\section{Bezout Intersection Domains}

Let $D$ be an integral domain. We say that some elements $a_1, a_2, \ldots, a_n \in D$ are \it pairwise incomparable \rm if $a_i \not \in (a_j)$ for every $i,j$. Through this article, when taking a collection of pairwise incomparable elements, we will generally assume $n \geq 2$.
The following definitions appear in the article \cite{Gue}. 

\begin{definition} 
\label{d1} An integral domain $D$ is called:
\begin{itemize}
\item  \it Bezout Intersection domain (BID) \rm if for any finite collection of pairwise incomparable elements $a_1, a_2, \ldots, a_n \in D$, the ideal $$I= (a_1) \cap (a_2) \cap \ldots \cap (a_n)$$ is either principal or  is not finitely generated.
\item  \it Strong Bezout Intersection domain (SBID) \rm if for any finite collection of pairwise incomparable elements $a_1, a_2, \ldots, a_n \in D$ with $n \geq 2$, the ideal $$I= (a_1) \cap (a_2) \cap \ldots \cap (a_n)$$ is not finitely generated.
\end{itemize}
 \end{definition}


\begin{remark} 
\label{r1}
 Let $D$ be an integral domain and consider the ideal $J=(a_1, a_2, \ldots, a_n)$. It is a standard exercise to see that $$J^{-1}:=(D : J)= \dfrac{1}{a} \left( \left( \frac{a}{a_1}\right) \cap \left( \frac{a}{a_2}\right) \cap \ldots \cap \left( \frac{a}{a_n}\right) \right) $$ where $a= \prod_{i=1}^n a_i.$ Hence Definition \ref{d1} can be equivalently restated saying: $D$ is BID if the inverse of every finitely generated ideal is either principal or not finitely generated. 
\end{remark}

Trivial examples of BID are GCD domains (in particular finite conductor or Noetherian BID are GCD domains).  In \cite[Proposition 2.5]{Gue} it is proved that a SBID needs to be local. Trivial examples of SBID are valuation domains, since in this case all the elements are comparable. As an easy consequence of a theorem by McAdam \cite{mcadam}, one can observe that integrally closed domains with linearly ordered prime ideals are SBID.

Using Remark \ref{r1}, it is easy to prove the following property of a Bezout Intersection domain.

\begin{lemma} 
\label{inve}
The invertible ideals of a Bezout Intersection domain $D$ are principal.
\end{lemma}

\begin{proof}
 Let $J=(a_1, \ldots, a_n)$ be an invertible ideal of $D$ and set $a=a_1 \cdots a_n$. The inverse of $J$ is $J^{-1}= \bigcap_{i=1}^n a_i^{-1}D$ and thus $a J^{-1}$ is a proper invertible ideal that is also the intersection of $n$ principal ideals of $D$. Since $D$ is a BID, then $a J^{-1}$ must be principal and therefore $J^{-1}$ and $J$ are also principal. 
\end{proof}

The following corollary derives from the last fact and generalizes the well-known fact that a Pr\"{u}fer GCD domain is a Bezout domain.

\begin{corollary} 
\label{c1}
Let $D$ be a Pr\"{u}fer BID. Then $D$ is a Bezout domain. 
\end{corollary}

An integral domain $D$ is a \it Pre-Schreier \rm domain if whenever an element $x$ divides $ab$, then $x=rs$ with $r$ dividing $a$ and $s$ dividing $b$. Pre-Schreier domains represent a natural and well-studied generalization of GCD domains.

 An ideal $I$ of an integral domain is said to be \it locally cyclic \rm if for every $x_1, x_2, \ldots, x_n \in I$, there exists $y \in I$ such that $ (x_1, x_2, \ldots, x_n) \subseteq (y) $. Equivalently, a locally cyclic ideal is the union of an ascending chain of principal ideals. We recall that this notion derives from that of locally cyclic module used in module theory and the word "locally" is not related with the localizations of the ring. In \cite[Theorem 1.1]{zafrullah}, Zafrullah shows that an integral domain is Pre-Schreier if and only if every non-zero ideal of the form $ (a_1) \cap (a_2) \cap \ldots \cap (a_n) $ is locally cyclic. A locally cyclic ideal is clearly finitely generated only if it is principal and therefore Pre-Schreier domains are BID. 

Examples of Pre-Schreier domains are directed unions of GCD domains. In \cite[Theorem 2.11]{Gue} the SBID are characterized  among the directed unions of Noetherian GCD domains. They consist exactly of the local domains having locally cyclic maximal ideal. This result shows that the rings known as \it quadratic Shannon extensions, \rm described in articles such as \cite{HLOST}, \cite{HOT}, \cite{GHOT}, \cite{FZ}, are non-trivial examples of SBID. For a concrete example that is not a valuation domain one may consider the ring $$ S= \bigcup_{n=0}^{\infty} k \left[ x,\dfrac{y}{x^n}, \dfrac{z}{x^n} \right]_{(x,\frac{y}{x^n}, \frac{z}{x^n})},
 $$ where $x,y,z$ are indeterminates over a field $k$.

\medskip

Easy examples of SBID domains that are not necessarily Pre-Schreier arise with the classical $k+\m$ construction. 

\begin{theorem} 
\label{k+m}
Let $k$ be a field and let $X$ be an indeterminate over $k$. Let $T$ be a local domain with maximal ideal $\m$ and containing its residue field $ k(X)=T/\m.$ 
The ring $D=k+\m$ is a SBID.
\end{theorem}

\begin{proof}
First we observe that any proper ideal $I$ of $T$ (including $\m$) it is not finitely generated as an ideal of $D$. Indeed, let $ \lbrace  x_l \rbrace_{l \in L} $ be a minimal set of generators for $I$ over $T$ and let $ \lbrace  f_h(X) \rbrace_{h \in H} $ be a set of generators of $k(X)$ as a $k$-algebra. Notice that $H$ is an infinite set. Thus, $I$ is generated as an ideal of $D$ by the elements of the infinite set $ \lbrace  x_lf_h(X) \rbrace_{l \in L, h \in H}. $ 

Let $I= (a_1) \cap (a_2) \cap \ldots \cap (a_n)$ for some pairwise incomparable elements $a_1, a_2, \ldots, a_n \in \m \subseteq D$.  Clearly $$I \subseteq J= a_1T \cap a_2T \cap \ldots \cap a_nT.$$ Take $f \in J$, hence $ \frac{f}{a_i} \in T $ for every $i=1,\ldots, s$. Now, if $ \frac{f}{a_i} \in \m $ for every $i$, we have $$ f= a_i \dfrac{f}{a_i} \in a_i\m \subseteq a_iD $$ and hence $f \in I$. When this happens for every $f \in J$, it follows that  $I=J$  is not finitely generated as an ideal of $D$. If instead, for some $f \in J$, we have $ \frac{f}{a_i} \in T \setminus \m $, then we have $ J= a_iT $ and hence $ \frac{a_i}{a_j} \in T $ for every $j$. By the assumption of having $a_i \not \in a_jD$ for every $j \neq i$, we need to know  that $ \frac{a_i}{a_j} = f_j(X) $ is a unit of $T$ and  is not in $D$ (hence  is a non-constant element of $k(X)$). Therefore, we consider the ideal $$ a_i^{-1}I = D \cap \bigcap_{j \neq i} g_j(X)D $$ where $ g_j(X)=f_j(X)^{-1} $. Such an ideal $ a_i^{-1}I $ is a proper ideal of $D$ since all the elements $ g_j(X),f_j(X) \not \in D $. Moreover, let $b \in \m$. For every $j$, we can write $b= g_j(X) (b f_j(X)) \in g_j(X)\m \subseteq g_j(X)D$. Hence in this case $ a_i^{-1}I = \m $ it is not finitely generated in $D$, and so neither is $I$.
\end{proof}

Taking the $k+\m$ construction where $\m$ is the maximal ideal of $T=k(X)[Y]_{(Y)}$, we get a non-Pre-Schreier SBID. Indeed in a Pre-Schreier domain irreducible elements are prime, while in $T$ elements such as $Y$ or $ Yf(X)$ are irreducible but not prime, since, for instance, $ (YX)(\frac{Y}{X}) \in (Y)$ but $ X, \frac{1}{X} \not \in k+\m $.

\medskip

We generalize the last theorem to the case of a pullback square of type $\square^{\ast}$ with local lower-left corner. For an extensive study about pullback construction in commutative ring theory see for example \cite{Gabelli-Houston, GH}. Our setting is the following: 
let $T$ be a local domain with maximal ideal $\m$ and let $B$ be an integral domain having quotient field $ \kappa:= \frac{T}{\m}.$ Let $\phi: T \to \kappa$ be the canonical surjective map. Define the ring $D:= \phi^{-1}(B)$ as in the pullback diagram:
\begin{center}
\begin{tikzcd}
   D \arrow[rightarrow]{r}\arrow[hookrightarrow]{d} 
  &  B \arrow[hookrightarrow]{d}
  \\ 
   T \arrow{r}
 &  \kappa
\end{tikzcd}
\end{center}
We recall some standard facts about this kind of pullback square with $T$ a local domain. 

\begin{lemma}
\label{lemmapulb} Assume the notation of the above pullback square. \begin{enumerate}
\item The ideal $\m$ is a divided prime ideal of $D$ (i.e. $\m \subseteq (x)$ for every $x \in D \setminus \m$).
\item For every $a,b \in D \setminus \m$, $\phi(a) \in \phi(b)B$ if and only if $a \in bD$.
\item Given a principal ideal $bB$ of $B$, $\phi^{-1}(bB)$ generates a principal ideal of $D$. 
\item Let $J$ be an ideal of $B$, then $J$ is finitely generated if and only if $\phi^{-1}(J)$ is finitely generated.
\end{enumerate}
\end{lemma}

\begin{proof}
(1) For $y \in \m$ and $x \in D \setminus \m$, $\phi(\frac{y}{x})= \phi(y)\phi(x)^{-1} = 0 \in B$ and thus $ \frac{y}{x} \in \m \subseteq D $. \\
(2) The ideal $\m$ is contained in every maximal ideal of $D$, and hence for every unit $u \in D$ and $y \in \m$, $u+y$ is a unit of $D$. We have $\phi(a)=\phi(b)\phi(c)$ if and only if there exists $y \in \m$ such that $a=bc+y= bc(1+\frac{y}{bc})$ if and only if $a \in bD$. \\
(3) It follows by \cite[Theorem 2.21, Proposition 2.22]{Gabelli-Houston} using the fact that $T$ is local. \\
(4) It follows by \cite[Proposition 2.14]{Gabelli-Houston}.
\end{proof}

\begin{theorem} 
\label{pullback}
Take the notation of the above pullback square.
The ring $D$ is BID (resp. SBID) if and only if $B$ is BID (resp. SBID).
\end{theorem}

\begin{proof} 
Let $I= (a_1) \cap (a_2) \cap \ldots \cap (a_n)$ for some pairwise incomparable elements $a_1, a_2, \ldots, a_n \in D$. Hence, there are only two possible cases: either $ a_1, a_2, \ldots, a_n \in \m $ or $ a_1, a_2, \ldots, a_n \in D \setminus \m $ and they all are non-units. \\
\bf Case 1: $ a_1, a_2, \ldots, a_n \in \m $. \rm \\
Let $x \in D \setminus \m$ a non-unit, if $I=xI$, clearly $I$ is not finitely generated as a consequence of Nakayama's Lemma. 
Hence, assume the opposite condition. Any element $z \in I$  can be always expressed in the form $z=a\gamma$ where $a= \prod_{i=1}^n a_i$ and $ \gamma \in (\frac{a}{a_1}, \frac{a}{a_2}, \ldots, \frac{a}{a_n})^{-1} $. Take $z \in I$ such that $\frac{z}{x} \not \in I$ and assume without loss of generality $\frac{z}{x} \not \in (a_1)$. It follows that $\frac{z}{a_1} \not \in (x)$ and thus $\gamma \frac{a}{a_1} \not \in \m$. 
Since the $a_i$ are incomparable and $\m$ is divided, this implies that $\gamma \frac{a}{a_i} \not \in \m$ for every $i$. But $\frac{a}{a_i} \in \m$, and this implies $\frac{a}{a_i} \in (\gamma \frac{a}{a_i})$ and hence $\gamma^{-1}= d \in D.$ Moreover, for every $i$, $ \frac{a}{a_id} \in T \setminus \m $ and hence $ \frac{a}{a_i} $ is associated to $d$ in $T$. It follows $a_iT = a_jT$ for every $i,j$, and hence $$ I= a_1 \left( D \cap u_2D \cap \ldots \cap u_nD \right)  $$ where $u_i$ are units of $T$. Since $T$ is local, by \cite[Proposition 1.7]{Gabelli-Houston}, $T=D_{\m}$ and hence $ u_i = \frac{h_i}{l_i} $ with $h_i,l_i \in D \setminus \m$. By multiplying a common factor, we get $$ I= w \left( (b_1) \cap (b_2) \cap \ldots \cap (b_n) \right) $$ where $b_i \in D \setminus \m$, and therefore we consider this intersection in the second case.  \\
\bf Case 2: $ a_1, a_2, \ldots, a_n \in D \setminus \m $. \rm \\
Using Lemma \ref{lemmapulb}(2), we observe that $$ \phi(I)= \bigcap_{i=1}^n \phi(a_i)B $$ and for some pairwise incomparable elements $b_1, b_2, \ldots, b_n \in B$, $$ \phi^{-1} \left( \bigcap_{i=1}^n b_iB \right) = \bigcap_{i=1}^n \phi^{-1}(b_i)D.$$ Clearly if $I$ is principal or not finitely generated, also $ \phi(I) $ is respectively principal or not finitely generated. For an ideal $ J=\bigcap_{i=1}^n b_iB $ principal or not finitely generated, we obtain that $\phi^{-1}(J)$ is respectively principal or not finitely generated using (3) and (4) of Lemma \ref{lemmapulb}. This concludes the proof.
\end{proof}

The previous theorem shows that the Bezout Intersection property, unlike   the Pre-Schreier property, has  bad behavior with respect to localizations. For a prime number $p$, consider the ring $D= \Z_{(p)} + \m$, where $\m$ is the maximal ideal of $T=\Q[X^2,X^3]_{(X^2,X^3)}$. By Theorem \ref{pullback}, $D$ is SBID, but $T=D_{\m}$ is clearly not BID.

In the next construction, we show that localization of a BID can fail also at a maximal ideal.

\begin{construction}
\label{construction} \rm 
Start with three indeterminates $x,y,z$ over a field $k$ and consider the ring $R=k[y^2,y^3,x]_{(y^2,y^3,x)}$. For $n \geq 1$, set $$ R_n = R \left[ \frac{y^2}{x^n}, \frac{y^3}{x^n}\right]_{(\frac{y^2}{x^n}, \frac{y^3}{x^n}, x)},$$ and call $D$ the directed union $D = \bigcup_{n\geq 1}^{\infty}R_n$. The ring $D$ is local with principal maximal ideal generated by $x$. Call $Q=\bigcap_{n\geq 1}^{\infty}x^nD$ the non-finitely generated prime ideal of $D$ adjacent to the maximal ideal.
The ring we want to study is defined as $$ A:= D(z) \cap D[z]_{(Q,z)}. $$ Elements such as $z+x$ are units in $A$ and therefore $A$ is a semilocal ring with two maximal ideals $\mathfrak{n}= xA$ and $\m=(Q,z)A$. 
Moreover, $A_{\mathfrak{n}}= D(z)$ and $$A_{\m}= A [x^{-1}] = D[z]_{(Q,z)} = k(x)[y^2, y^3, z]_{(y^2, y^3, z)}. $$
In the next theorem we show that $A$ is a BID. It is clear that its localization $S:= A_{\m}$ is not a BID since the intersection $y^2S \cap y^3S = (y^5,y^6)S$. 
\end{construction}

\begin{lemma} 
\label{l11}
Let $A$ be the ring defined in Construction \ref{construction}. Call as above $S:= A_{\m}$ and consider also its integral closure $\overline{S}= S[y]$. Let $b \in \overline{S}$ and $c \in A \setminus \mathfrak{n}$, and assume $bc \in A$. It follows that $b \in A$.
\end{lemma}

\begin{proof} 
Use the fact that $ \overline{S}=A[x^{-1}, y] $ to express $b$ as $$ b= a + a_0y + \sum_{i \geq 1} a_ix^{-i} $$ where $a,a_i \in A$, $a_0 \not \in Q$, $a_i \not \in (x)$, and only finitely many $a_i$ are nonzero. The element $bc \in A$ if and only if $ca_0y \in A$ and also each summand $ca_ix^{-1} \in A$, but this is impossible since $(A :_{A} y)=Q$, $(A :_{A} x^{-n})= (x^n)$ and $c \not \in \mathfrak{n}$ which is a prime ideal. Hence, the only possible way to have $bc \in A$ is to have $a_i = 0$ for every $i \geq 0$ and therefore $b \in A.$ 
\end{proof}

\begin{theorem} 
\label{localization}
The ring $A$ defined in Construction \ref{construction} is a BID.
\end{theorem}

\begin{proof} 
First let $a \in \mathfrak{n} \setminus \m$, assume without loss of generality $a= x^n$ and let $b \in A$ be a non-unit. We have the following possible cases.
If $b \in \mathfrak{n} \setminus \m$ or $b \in \m\mathfrak{n}$ but $x^n$ divides $b$, $a$ and $b$ are comparable.
If $b \in \m \setminus \mathfrak{n}$ then $a$ and $b$ are comaximal and $(a) \cap (b)= (ab)$. If instead
 $b \in \m\mathfrak{n}$ and the largest power of $x$ dividing $b$ is $k$ with $1 \leq k < n$, then $b=x^kc$ for some $c \in \m \setminus \mathfrak{n}$ and thus $ (a) \cap (b) = x^k ((x^{n-k}) \cap (c))) $ is principal and contained in $\m\mathfrak{n}$.

From these facts, in order to study a finite intersection of principal ideals of $A$, we can reduce to considering only elements in $\m$. Hence let $I= (a_1) \cap (a_2) \cap \ldots \cap (a_n)$ for some pairwise incomparable elements $a_1, a_2, \ldots, a_n \in \m$. Such elements are all non-units in the overring $\overline{S} = S[y]$ that is a regular local ring, hence a UFD. It follows that $$I \subseteq a_1\overline{S} \cap a_2\overline{S} \cap \ldots \cap a_n\overline{S} = \left( \frac{a}{s} \right) \overline{S}=: w\overline{S} $$ where $a= a_1 a_2 \cdots a_n$ and $s \in \overline{S}$ is the gcd of $a_1, a_2, \ldots, a_n$ in $\overline{S}$. We consider now three possible cases: \\
\bf Case 1: \rm There exists $i \in \lbrace 1,2, \ldots, n \rbrace$ such that $ \frac{w}{a_i} \in \overline{S} \setminus A$. \\
The ideal $Q$ is generated by all the elements of the form $ \frac{y^2}{x^n}, \frac{y^3}{x^n} $ for every $n \geq 0.$ Hence, since $\overline{S}= A[x^{-1}, y]$, it is clear that $Q\overline{S} \subseteq A$. It follows that for every element $q \in Q$, $qw\in I$. Furthermore since $ \frac{w}{a_i} \not \in A$, clearly $w \not \in I$. This makes $I$ not finitely generated since $Q$ is not finitely generated. \\
\bf Case 2: \rm  The ideal $ J= (\frac{w}{a_1}, \frac{w}{a_2}, \ldots, \frac{w}{a_n}) \subseteq Q.$ \\
In this case $w \in I$ and since $Q\overline{S} \subseteq A$, the element $w(yx^{-n}) \in I$ for every $n \geq 0$. These elements form an infinite set and they are all needed as generators for $I$.  \\
\bf Case 3: \rm  The ideal $ J= (\frac{w}{a_1}, \frac{w}{a_2}, \ldots, \frac{w}{a_n}) \subseteq A \setminus Q. $ 
\\ Also in this case $w \in I$ and every element of $I$ is of the form $ wb$ with $b \in \overline{S}.$ If $J \subseteq (x^n)$ for some $n$ (maximal with respect to this property), we get $wx^{-n} \in I$ and hence, by eventually replacing $w$ by $wx^{-n}$, we may assume $J \nsubseteq (x)$.
Now assume $ c:= \frac{w}{a_1} \not \in (x) $ and take $wb \in I$ with $b \in \overline{S}$. It follows that $bc \in A$. By Lemma \ref{l11}, this implies $b \in A$ and hence $I$ is principal generated by $w$. 
\end{proof}

\begin{remark} 
\label{r11}
The ring $A$ described above shows that it is possible to have examples of intersections of three principal ideals that are principal while the intersection of two of them is not finitely generated.
We want to compute the intersection $$ I:= (y^2z) \cap (y^3(z-y^2)) \cap (y^5(z-y^3)). $$ Following the proof of Theorem \ref{localization}, in this case $w:= y^5 z (z-y^2)(z-y^3)$ and $$  \frac{w}{y^5(z-y^3)}= z (z-y^2) \in A \setminus Q. $$ Hence by Case 3, $I=(w)$.
But $$
(y^2z) \cap (y^3(z-y^2)) =  \\ z(z-y^2) \left( y^5, y^6, \frac{y^5}{x}, \frac{y^6}{x}, \frac{y^5}{x^2}, \frac{y^6}{x^2}, \ldots  \right)    $$ is not finitely generated.
\end{remark}

In the counterexamples shown above one can observe that localization fail to be a BID when made with respect to a prime ideal which is not a maximal $t$-ideal. We recall the definition of the usual star operations $v$ and $t$. For an integral ideal $I$ of $D$, $I_v= (I^{-1})^{-1}$ and $$ I_t= \bigcup_{\substack{J \subseteq I \\ J \, f.g. }} J_v.$$ An ideal $I$ is \it divisorial \rm if $I=I_v$ and is a \it $t$-ideal \rm if $I=I_t$. There always exist ideals maximal with respect to the property of being $t$-ideals, they are called \it maximal $t$-ideals \rm and they are prime ideals. 
It is still uncertain whether the localization of a BID with respect to a maximal $t$-ideal is a BID. We leave this as a question:

\begin{question} 
Let $D$ be a Bezout Intersection domain and let $P$ be a maximal $t$-ideal of $D$. Is $D_P$ a BID?
\end{question}

In next section we will show that the answer is yes if assuming $D$ integrally closed and $P$ a well-behaved $t$-ideal (i.e. $PD_P$ is also a $t$-ideal). For this, see Remark \ref{kaplansky} and Theorem \ref{carattsbid}.
For references about well-behaved $t$-ideals, the reader may consult \cite{z3}, \cite{ACZ}.

\medskip

To conclude this section, we describe an example of a BID Krull domain.
In Corollary \ref{c1} it is shown that a Pr\"{u}fer BID has to be a Bezout domain. One may naturally ask if a Bezout Intersection PvMD is necessarily a GCD domain. The answer is no, here we provide an example of a BID non-GCD domain that is also a Krull domain, hence a PvMD.
To this purpose, we consider the ring described in \cite[Example 2.5]{glaz1} as a standard example of a Krull domain that is not a finite conductor domain. Let $ \lbrace X_n \rbrace_{n \in \N} $ be a countable collection of indeterminates over a field $K$, set $$ D=K[\lbrace X_iX_j \rbrace_{i,j \in \N}], $$ and call $Q(D)$ the quotient field of $D$. The ring $D$ is a Krull domain and, setting $ R=K[\lbrace X_n \rbrace_{n \in \N}] $, we can express $D$ as the intersection $D=R \cap Q(D)$. We want to prove that $D$ is a BID. 

As a $K$-vector space, $D$ is the subspace of $R$ generated by the monomials of even degree. We denote by $A$ the $K$-vector subspace of $R$ generated by the monomials of odd degree (this vector space is clearly not a ring), and for every $f \in R$, we write $f= f_p + f_d$ where $f_p \in D$ and $f_d \in A$. 


\begin{lemma} 
\label{krulllemma}
Let $h \in R$ be an irreducible polynomial such that $h_p,h_d \neq 0$. Then, there exists a unique 
polynomial $h^{\prime} \in R$, such that $hh^{\prime} \in D$ and $ h^{\prime} $ is the minimal polynomial  with  respect to divisibility satisfying this condition. 
\end{lemma}

\begin{proof} Assume $hg \in D$ for some $g \in R$. Clearly $g_p,g_d \neq 0$ and $$hg= h_pg_p + h_dg_d + h_pg_d + h_dg_p \in D$$ implying $h_pg_d = - h_dg_p$. Since $h$ is irreducible, $h_p$ and $h_d$ have no common factors in the UFD $R$ and hence necessarily $g_p$ is a multiple of $h_p$. The minimal possible choice for a such $g$ is obtained setting $ g_p:=h_p $ and thus $g_d= -h_d$. We conclude defining $h^{\prime}:= h_p - h_d$. 
\end{proof}

\begin{theorem} 
\label{krulldomain}
The ring $ D=K[\lbrace X_iX_j \rbrace_{i,j \in \N}] $ is a BID.
\end{theorem}

\begin{proof} Let $f_1, f_2, \ldots, f_s \in D$ be pairwise incomparable elements and set $$I= (f_1) \cap (f_2) \cap \ldots \cap (f_s).$$ The ideal $$ f_1R \cap f_2R \cap \ldots \cap f_sR = \lambda R $$ is principal generated by $ \lambda = \mbox{lcm}(f_1, f_2, \ldots, f_s) $. This polynomial is expressible in a unique way as $ \lambda= \frac{f}{g} $ where $f=f_1 \cdots f_s$ and $g$ is a product involving the common factors of the pairs $f_i, f_j$. 

If $h$ is a common irreducible factor of $f_i$ and $ f_j$, and $h_p,h_d \neq 0$, by Lemma \ref{krulllemma}, also $ h^{\prime} $ has to be a common factor of $f_i$ and $ f_j$ and therefore $ hh^{\prime} $ divides $g$. It follows that we can express $g= g_1 \cdots g_e h_1 \cdots h_c$ with $g_i \in D$ and $h_i \in A.$ Since $f \in D$, it is easy to observe that either $ \lambda \in D $ or $ \lambda \in A $.

First we consider the case $ \lambda \in D $. Clearly also $ \frac{\lambda}{f_i} \in D $ for every $i$ and hence $\lambda \in I$. Every $\alpha \in I $ can be written as $ \alpha = \lambda \frac{\alpha}{\lambda} $, moreover $I \subseteq \lambda R$ and thus $ \frac{\alpha}{\lambda} \in R \cap Q(D) = D $. It follows that $I= (\lambda)$ is principal.

In the second case, $\lambda, \frac{\lambda}{f_i} \in A$ and $ \frac{\lambda}{f_i} X_n \in D $ for every $i$ and $n$. It follows that $I$ contains the element $\lambda X_n$ for every $n \in \N$ and therefore  cannot be finitely generated.
\end{proof}


\section{The $\xi$ operation}
In this section we study Bezout and Strong Bezout Intersection domains in term of an operation on ideals, closely related to the well-known star operation $w$.
A nonzero finitely generated integral ideal $I$ of an integral domain $D$ is a \it Glaz-Vasconcelos ideal \rm if $I^{-1}=D$ (equivalently if $I_v=D$)\cite{GV}. The set of all Glaz-Vasconcelos ideals of $D$ is denoted by $GV(D)$. Let $K$ be the quotient field of $D$ and let $I\subseteq D$ be an ideal.
The star operation $w$ is defined as $$ I_w = \lbrace x \in K \mbox{ : } xJ \in I, \mbox{ for some } J \in GV(D) \rbrace. $$ We want to define a similar operation on ideals involving the use of trace ideals. An ideal $J \subseteq D$ is a \it trace ideal \rm if $J=II^{-1}$ for some nonzero ideal $I \subseteq D$. Sometimes we will use the notation $\Tr(I)=II^{-1}$. Glaz-Vasconcelos ideals are always trace ideals, since if $J \in GV(D)$, $J=\Tr(J).$


\begin{definition} 
\label{d2} Let $I$ be an ideal of $D$, we define $$ I_{\xi} = \lbrace x \in K \mbox{ : } xJ \subseteq I, \mbox{ for some } J \mbox{ a finitely generated trace ideal} \rbrace. $$
 \end{definition}

\begin{prop} Let $I,I_1$ be ideals of $D$. The operation $\xi$ defined above has the following properties:
\label{operation}
\begin{enumerate}
\item $ I_{\xi} $ is an ideal of $D$.
\item For $x \in D$, $xI_{\xi}= (xI)_{\xi}$.
\item For $I \subseteq I_1$, $ I_{\xi} \subseteq (I_1)_{\xi} $.
\item $I \subseteq I_{\xi}$.
\item $I_w \subseteq I_{\xi}$.
\end{enumerate}
Moreover the following assertions are equivalent: \begin{enumerate}
\item[(i)] Every finitely generated trace ideal is Glaz-Vasconcelos.
\item[(ii)] $\xi=w$.
\item[(iii)] $\xi$ is a star operation.
\end{enumerate}
\end{prop}

\begin{proof}
(1) Take $x,y \in I_{\xi}$. Then there exist two finitely generated trace ideals $J_1= \Tr(I_1)$ and $J_2= \Tr(I_2)$ such that $ xJ_1, yJ_2 \subseteq I $. It is easy to observe that, in general $$ \Tr(I_1I_2) \subseteq \Tr(I_1)\cap \Tr(I_2) $$ (cf. \cite[Proposition 1.4]{HHS}). Hence $$(x+y)\Tr(I_1I_2) \subseteq xJ_1 + yJ_2 \subseteq I$$ implying $x+y \in I_{\xi}$. It is now straightforward to prove that $ I_{\xi} $ is an ideal. \\
(2) and (3) are clear consequences of the definitions. \\
(4) Just observe that $D$ itself is a finitely generated trace ideal. \\
(5) Follows from the fact that Glaz-Vasconcelos ideals are finitely generated trace ideals. \\
We prove now the equivalence of the three conditions (i)-(ii)-(iii). The implications (i) $ \Rightarrow $ (ii) and (ii) $\Rightarrow $ (iii) are trivial. Assume (iii), then $D_{\xi}=D$. By way of contradiction, suppose there exists a finitely generated trace ideal $J\not \in GV(D)$. It follows that there exists $z \in J^{-1} \setminus D$, hence $zJ \subseteq D$ implying $z \in D_{\xi} \setminus D$ and this is a contradiction.
\end{proof}

\begin{remark}
It follows from the above proposition that $\xi$ may fail to be a star operation.   In general, it may also fail  to be a semistar operation. In the ring $D=k[X^2,X^3]_{(X^2,X^3)}$, 
the maximal ideal $\m=\Tr(\m)$ is a finitely generated trace ideal. Since $$X^2\m, X^3\m \subseteq (X^4, X^5),$$ we get $(X^4, X^5)_{\xi} = \m$. But $\m_{\xi}= D[x]$ and therefore $(I_{\xi})_{\xi} \neq I_{\xi}$ for some ideal $I$.
However, setting $I_1=I$ and for $i \geq 2$, $I_n:= (I_{n-1})_{\xi}$, the operation $\star$ defined by $$I^{\star}= \bigcup_{i=1}^{\infty} I_n$$ is a proper semistar operation.
\end{remark}

\begin{remark} \label{kaplansky}
When $D$ is integrally closed, $\xi= w$ is a star operation. This is a consequence of a well-known fact appearing as exercise in Kaplansky's book. Indeed, in an integrally closed domain, $J= II^{-1}$ finitely generated implies $J^{-1}=D$ (cf. \cite[exercise 39, pag.45]{Kap}).
\end{remark}

\medskip




The next result describes some properties of Bezout Intersection domains and relate them with the operation $\xi$.

\begin{theorem} 
\label{3properties}
Let $D$ be a Bezout Intersection domain. Then:
\begin{enumerate}
\item[(a)] $\xi=w$.
\item[(b)] Invertible ideals of $D$ are principal.
\item[(c)] Finitely generated maximal $t$-ideals of $D$ are principal.
\end{enumerate}
 \end{theorem}

\begin{proof}
(a) Let $J=II^{-1}$ be a finitely generated trace ideal. It follows that $I$ and $I^{-1}$ are both finitely generated. By Remark \ref{r1}, $I^{-1}$ is principal. Hence $J=aI$ and $$J^{-1}=a^{-1}I^{-1}=a^{-1}aD=D.$$ This proves $J \in GV(D)$ and Proposition \ref{operation} concludes the proof. \\
(b) This is proved in Lemma \ref{inve}. \\ 
(c) A finitely generated maximal $t$-ideal $P$ is divisorial. Hence $P= D \cap \bigcap_{l\in L}x_lD$. Take $x \in \lbrace x_l\rbrace_{l \in L}$. It follows that $ P \subseteq D \cap xD$, which  is a proper $t$-ideal. Hence $P=D \cap xD$ is the intersection of two principal ideals and it has to be principal since $D$ is BID.
\end{proof}


We do not know in general, whether the converse of Theorem \ref{3properties} is true, that is, whether any integral domain satisfying conditions (a),(b),(c) is a BID. We leave this as a question:

\begin{question} 
Do conditions (a),(b),(c) of Theorem \ref{3properties} imply that $D$ is BID? 
\end{question}

In some cases the answer is yes.
For a Noetherian domain $D$, condition (c) of Theorem \ref{3properties} it is sufficient to imply that $D$ is a UFD and thus a BID.
It is possible to prove that the converse of Theorem \ref{3properties} holds also for integral domains such that the star operation $w$ is the identity. These domains are called \it DW domains \rm and they have been widely studied in articles such as \cite{mimouni} and \cite{PT}. In a DW domain $D$, $J \in GV(D)$ implies $J_w=D$, thus the only Glaz-Vasconcelos ideal is $D$ itself and every finitely generated proper ideal is contained in a divisorial ideal. For this reason, DW domains are exactly the integral domains in which every maximal ideal is a $t$-ideal. 

Before proving this result, we briefly discuss the fact that conditions (a),(b),(c) are not related one to each other. Indeed for any two of them, there are integral domains satisfying both but not satisfying the third one. Indeed any local Noetherian integrally closed domain that is not a UFD satisfies (a),(b) but not (c) (consider for instance $k[X,\frac{Y}{X},\frac{Y^3}{X^4}]_{(X,\frac{Y}{X},\frac{Y^3}{X^4})}$).

Any non-Bezout Pr\"{u}fer domain having all the maximal ideals non finitely generated satisfies (a),(c) but not (b) (for instance the classical integer-valued polynomials ring Int$(\mathbb{Z})$).

Finally, as example of integral domain satisfying (b),(c) but not (a), we may take the ring $$ D=k[y, \sqrt[3]{y^2}, \sqrt[9]{y^2}, \sqrt[27]{y^2}, \ldots ]_{\mathfrak{M}} $$ where $\mathfrak{M}$ is the ideal generated by $y$ and by $ \lbrace\sqrt[3^n]{y^2} \rbrace_{n \geq 1} $. This ring was constructed by Hochster, but using a different notation, in order to find a one dimensional local non-Noetherian coherent domain 
(see \cite[Section 7, pag. 278]{glaz2}). All the ideals of the form $ (y, \sqrt[3]{y^2}, \ldots, \sqrt[3^n]{y^2}) $ are finitely generated trace ideals not Glaz-Vasconcelos.  

We prove now that conditions (b) and (c) are sufficient to force a DW domain to be a BID. To introduce the next result, we recall that the identity star operation is usually called $d$.

\begin{theorem} 
\label{dw} Let $D$ be a DW domain. The following conditions are equivalent:
\begin{enumerate}
\item $D$ is BID.
\item The invertible ideals of $D$ are principal and $\xi=w$.
\end{enumerate}
\end{theorem}

\begin{proof} The first implication follows by Theorem \ref{3properties}. For the second, observe that in a DW domain such that $\xi=w$, then $\xi=d$ and therefore the only finitely generated trace ideal of $D$ is $D$. Let $J$ be a finitely generated ideal and assume $J^{-1}$ to be also finitely generated. It follows that $JJ^{-1}$ is finitely generated and hence $JJ^{-1}= D$ and $J$ is invertible. By assumption $J$ and $J^{-1}$ are principal and $D$ is a BID as a consequence of Remark \ref{r1}.
\end{proof}

Local DW domains are called \it $t$-local domains \rm and they are the local integral domains such that the maximal ideal is a $t$-ideal. Properties of $t$-local domains and their connections and differences with valuation domains are surveyed in \cite{FZ}. We can now give a useful characterization of the SBID. 

\begin{theorem} 
\label{carattsbid} Let $D$ be an integral domain. The following conditions are equivalent:
\begin{enumerate}
\item $D$ is SBID.
\item $D$ is $t$-local and BID.
\item $D$ is $t$-local and $\xi=w$.
\item $D$ is local and $\xi=d$.
\end{enumerate}
\end{theorem}

\begin{proof} 
(1)$ \Rightarrow $(2) We only need to prove that a SBID is $t$-local. Let $\m$ be the maximal ideal of $D$ and let $J=(a_1, a_2, \ldots, a_n) \subseteq \m$ be a finitely generated ideal (the elements $ a_1, a_2, \ldots, a_n $ are pairwise incomparable). Assume by way of contradiction $J_v= D$, thus $n \geq 2$ and $J^{-1}=D$. As a consequence of Remark \ref{r1}, the intersection $(a_1)\cap (a_2) \cap \ldots \cap (a_n)$ is principal. This contradicts the assumption that $D$ is SBID. \\
(2)$ \Rightarrow $(3) Follows by Theorem \ref{3properties}. \\
(3) $ \Leftrightarrow $ (4) is due to the fact that $t$-local domains are the local domains in which $w=d$, and to item 5 of Proposition \ref{operation}. \\
(4)$ \Rightarrow $(1) We apply Theorem \ref{dw} using the fact that, in a local domain invertible ideals are principal. Hence $D$ is a BID, but since $\xi=d$, the inverse of a finitely generated non-principal ideal $I$ cannot be finitely generated otherwise $II^{-1}$ would be a proper finitely generated trace ideal. This makes $D$ a SBID.
\end{proof}

As a comment to this theorem we remark that the ring $$ D=k[y, \sqrt[3]{y^2}, \sqrt[9]{y^2}, \sqrt[27]{y^2}, \ldots ]_{\mathfrak{M}}, $$ mentioned above, shows that the condition $\xi=w$ is necessary even for a one-dimensional local domain with non-finitely generated maximal ideal in order to be a SBID. 

However, without using the assumption $\xi=w$, it is possible to prove that if the maximal $t$-ideals of an integral domain $D$ are locally cyclic and the invertible ideals of $D$ are principal, then $D$ has to be a BID. We need a preliminary lemma extending to Bezout Intersection domains a well-known statement true for UFDs and GCDs, and that makes use of the fact that finitely generated locally principal ideals of an integral domain are invertible.


\begin{lemma} 
\label{locally}
Let $D$ be an integral domain such that $D_M$ is BID for every maximal ideal of $D$ and every invertible ideal of $D$ is principal. Then $D$ is a Bezout Intersection domain.
\end{lemma}

\begin{proof}
Let $a_1, a_2, \ldots, a_n \in D$ and set $I= a_1D \cap a_2D \cap \ldots \cap a_nD$. Assume that $I$ is finitely generated and let $M$ be a maximal ideal of $D$. We have that $ID_{M}= a_1D_M \cap a_2D_M \cap \ldots \cap a_nD_M$ is also finitely generated. Hence $ID_M$ is principal since $D_M$ is a BID and therefore $I$ is invertible because it is finitely generated and locally principal. By assumption $I$ is principal and hence $D$ is a BID.
\end{proof}

\begin{theorem} 
\label{locallycyclic} Let $D$ be an integral domain and assume that every maximal $t$-ideal of $D$ is locally cyclic and every invertible ideal of $D$ is principal. Then $D$ is a Bezout Intersection domain. 
\end{theorem}

\begin{proof} A locally cyclic ideal can be expressed as a directed union of principal ideals. Let $P= \bigcup_i x_iD$ be a maximal $t$-ideal and let $M$ be a maximal ideal of $D$ containing $P$. Thus $PD_M= \bigcup_i x_iD_M$ is a locally cyclic (possibly principal) ideal and hence a $t$-ideal.
If $ID_M$ is a $t$-ideal of $D_M$, by \cite[Lemma 3.17]{kang}, $I$ is a $t$-ideal of $D$, and therefore every maximal $t$-ideal of $D_M$ is locally cyclic. In light of the result of Lemma \ref{locally}, we can reduce to assuming that $D$ is local with maximal ideal $\m$.
Let 
$I = (a_1)\cap (a_2) \cap \ldots \cap (a_n)$ for some pairwise incomparable elements $a_1, a_2, \ldots, a_n$. 
Let $ \alpha  \in I $ and set $$J= \left( \frac{\alpha}{a_1}, \frac{\alpha}{a_2}, \ldots, \frac{\alpha}{a_n} \right).$$ If $J_v=D$, we have $$D= J^{-1}= \alpha^{-1}I$$ and this implies $I$ is principal. Otherwise $J$ is contained in some maximal $t$-ideal and therefore $J \subseteq (x) \subseteq \m$, implying $\frac{\alpha}{x} \in (a_i)$ for every $i$. It follows that, if $I$ is not principal, for every $ \alpha \in I $, there exists $ x \in \m $ such that $ \frac{\alpha}{x} \in I. $ Hence $I = \m I$ is not finitely generated by Nakayama's Lemma and $D$ is a BID.
\end{proof}

\begin{corollary} \label{corsbid}
A local domain $D$ with locally cyclic maximal ideal $\m$ is a Strong Bezout Intersection domain.
\end{corollary}

\begin{proof} The invertible ideals of a local domain are principal and a locally cyclic ideal is a $t$-ideal. By Theorem \ref{locallycyclic}, $D$ is a BID. Hence by Theorem \ref{carattsbid} is a SBID.
\end{proof}

In \cite{PT}, Park and Tartarone define the class of \it GCD-Bezout domains \rm as the integral domains in which the existence of the gcd of a finite set of elements is equivalent to the existence of the Bezout identity for those elements. They show that $D$ is a GCD-Bezout domain if and only if there do not exist proper primitive ideals in $D$. We recall that an ideal is \it primitive \rm if it is not contained in any proper principal ideal \cite{arnold}. Hence, GCD-Bezout domains are exactly the integral domains such that every maximal ideal is locally cyclic. As an application of Theorem \ref{locallycyclic}, we get the following corollary.

\begin{corollary} 
\label{gcdbezout}
A GCD-Bezout domain is BID if and only if every invertible ideal is principal.
\end{corollary}

As a consequence of Corollary \ref{corsbid} and Theorem \ref{carattsbid}, we can describe different classes of integral domains between $t$-local domains and valuation domains. Indeed, the following (non-reversible) implications hold for an integral domain: 
$$ \mbox{Valuation}\Rightarrow  \mbox{ Local with locally cyclic maximal ideal } \Rightarrow \mbox{ SBID } \Rightarrow t-\mbox{local}.$$
Using Theorem \ref{locallycyclic}, it is possible to prove that the last three conditions are equivalent for a Pre-Schreier domain.
Hence, this gives a characterization of the Pre-Schreier Strong Bezout Intersection domains.

\begin{corollary} 
\label{preschreier} Let $D$ be a Pre-Schreier local domain with maximal ideal $\m$. The following assertions are equivalent:
\begin{enumerate}
\item $\m$ is locally cyclic.
\item $D$ is a SBID.
\item $D$ is $t$-local.
\end{enumerate}
\end{corollary}

\begin{proof} In light of the above discussion it is sufficient to prove that the maximal ideal of a $t$-local Pre-Schreier domain is locally cyclic. Let $J=(a_1, \ldots, a_n) \subseteq \m$, set $a=a_1a_2\cdots a_n$ and $$I = \left( \frac{a}{a_1} \right) \cap \left( \frac{a}{a_2} \right) \cap \ldots \cap \left( \frac{a}{a_n} \right).$$ Since $D$ is Pre-Schreier, $I$ is locally cyclic and since $J_v \subseteq \m$, necessarily $I \neq (a)$. This implies that $a \in (x) \subseteq I$. We can write $x=a\gamma$ for some $\gamma \in J^{-1} $, but $ a \in (x) $ implies that $\gamma^{-1}= d \in \m$. It follows that $d^{-1}J \subseteq D$ and therefore $J \subseteq (d) \subseteq \m$ and hence $\m$ is locally cyclic.
\end{proof}

A local domain with locally cyclic maximal ideal does not need to be Pre-Schreier, as shown by the ring $D= \Z_{(p)} + \m$, where $\m$ is the maximal ideal of $T=\Q[X^2,X^3]_{(X^2,X^3)}$ and $p$ is a prime number.


\section{Construction of Strong Bezout Intersection domains}

In this section we introduce a polynomial-type construction that is perhaps most similar to the construction of a Rees ring.  We start with a local integral domain $R$ 
and we choose an ideal $\I$ of $R$.  We then introduce a collection $\F$, generally infinite, of indeterminates over $R$.  The elements of $\F$ can have multiplicative
relations with each other.  I.e.  if $f$ and $g$ are in $\F$ it may be that $fg$ is also in $\F$.    We then essentially treat elements of the form 
$fx$ where $x$ is in the chosen ideal $\I$ and $f \in \F$ as being variable in a polynomial-type ring.  This process will often yield interesting examples.  
Results vary widely according to what type of ideal is chosen and what multiplicative properties that set $\F$ has.

\begin{definition} \label{almostmultiplicative}
Let $R$ be an integral domain with quotient field  $K$ and consider an infinite set $\F=\lbrace f_i \rbrace_{i \in \Lambda}$ such that: 
\begin{itemize}
\item Each $f_i$ is transcendent over $K$. 
\item For every finite subset $ \h \subseteq \F  $, there exists an infinite subset $ \G \subseteq \F $ such that $hg \in \F$ for every $h \in \h$ and $g \in G$.
\end{itemize}
We refer to a set $\F$ satisfying these conditions as an \it almost multiplicatively closed transcendent set \rm over $R$ (clearly a properly multiplicatively closed set is also almost multiplicatively closed in this sense).
\end{definition}

\begin{definition} \label{RIF}
Let $R$ be a local domain with maximal ideal $\m$ and let $\I \subseteq \m$ be a proper ideal of $R$. Let $\F$ be an almost multiplicatively closed transcendent set over $R$ and set $$ \I \F = \lbrace if \mbox{ : } i \in \I, \, f \in \F \rbrace. $$ We define the local ring $$ R(\I,\F):= R[\I \F]_{\mathfrak{M}} $$ where $ \mathfrak{M} $ is the ideal generated by $\m$ and $\I \F.$
\end{definition} 

\begin{example} \label{exc}
We list some relevant examples of this construction:
\begin{enumerate}
\item An easy example is obtained taking $ \F = \lbrace x^n \rbrace_{n \geq 1} $ where $x$ is an indeterminate over $K$. This class of examples includes rings such as $$ k[a,ax,ax^2, ax^3, \ldots]_{(a,ax,ax^2, ax^3, \ldots)} $$ (here $R=k[a]_{(a)}$ and $\I=\m.$)
\item Also the $k+\m$ constructions can be obtained as particular examples of this construction. If $R$ is a local domain containing its residue field $k$, given an indeterminate $X$ over $k$, we set $\F $ to be a basis of $k(X)$ as a $k$-vector space and choose $\I=\m$. Clearly $\F$ is multiplicatively closed and $ R(\I,\F) $ is a standard $k+\m$ construction. Choosing different ideals $\I \subsetneq \m$, one may produce many variations of the same construction, often not integrally closed.
\item A different example in which $\F$ is not multiplicatively closed but only almost multiplicatively closed, is the following. Given a countable set of indeterminates $ \lbrace x_n \rbrace_{n \geq 1} $ over $K$, define the set $$ \F = \lbrace  x_{i_1}x_{i_2} \cdots x_{i_r}, \mbox{ squarefree monomial} \rbrace. $$ Clearly, for every $f \in \F$ and $m \geq 2$, $f^m \not \in \F$ but still $\F$ fulfills the condition of Definition \ref{almostmultiplicative}.
\end{enumerate}
\end{example}

When the ideal $\I$ is chosen to be $\m$-primary, as in some of the examples described above, 
the ring $ R(\I, \F) $ turns out to be a SBID. We prove this fact now.

\medskip

Consider the following notation.
We define $\pi: R(\I, \F) \to R$ to be the canonical surjection whose kernel is the ideal $\I \F \subseteq R(\I, \F)$. Each element of $R(\I,\F)$ is of the form $$ d=r_0 + r_1g_1 + \ldots + r_lg_l $$ where $r_0= \pi(d) \in R$,  $r_i \in \I$ for $i \geq 1$, and $g_i= \prod^{s_i}_{j=1} f_{ij}$ are products of elements of $\F$ for which $s_i \leq k$ where $k$ is the largest power such that  $r_i \in \I^{k}$ ($k$ may also be $ \infty $). We call an element of the form $r_ig_i$ (or $r_0$) a \it monomial \rm of $d$.

\medskip

 Call $\h (d)$ the set of all such elements $ f_{ij} \in \F $ dividing the monomials $g_i$ of $d$ in the overring $R[\F]$. Observe that $\h (d)$ is a finite subset of $ \F. $

\begin{lemma} \label{lemmino}
Let $R$ be a local domain with maximal ideal $\m$ and let $\I \subseteq \m$ be an $\m$-primary ideal of $R$. Given $r_1, r_2, \ldots, r_n \in \m$, there exists $t \in R \setminus \I $ such that $r_it \in \I$ for every $i=1, \ldots, n.$
\end{lemma}

\begin{proof} Since $\I$ is $\m$-primary, there exist minimal integers $e_i$ such that $r_i^{e_i} \in \I$. Consider the set $L= \lbrace \prod_{i=1}^{n} r_i^{h_i} \mbox{ : } 0 \leq h_i \leq e_i \rbrace $. This set is finite, partially ordered with respect to the order relation induced by divisibility and its maximal elements are in $\I$. Hence it is possible to find some element in $L$, maximal with respect to the property of not belonging to $\I$. Take $t$ as one of such elements (possibly $t=1$ if already $r_1, r_2, \ldots, r_n \in \I$).
\end{proof}

\begin{theorem} \label{primary}
Let $R$ be a local domain with maximal ideal $\m$ and let $\I \subseteq \m$ be an $\m$-primary ideal of $R$. Then the ring $ D= R(\I,\F) $ is a Strong Bezout Intersection domain.
\end{theorem}

\begin{proof} Let $I= (a_1) \cap (a_2) \cap \ldots \cap (a_n)$ for some pairwise incomparable elements $a_1, a_2, \ldots, a_n \in D$. Let $a \in I$ be an element obtained by taking the product $ a_1a_2 \cdots a_n $ and, if needed, dividing some common factor of the $a_i$ in order to make $a$ a minimal generator of $I$ (we may assume by way of contradiction $I$ to be finitely generated). Define for $i=1, \ldots, n$, $$ b_i:=\frac{a}{a_i}. $$
Since $b_i \in D$, set $$ \h := \bigcup_{i=1}^n \h(b_i) $$ and observe that $ \h $ is a finite subset of $\F$ and by Definition \ref{almostmultiplicative}, there exists an infinite set $  \G \subseteq \F $ such that $\h \G \subseteq \F$. Also set $r_i=\pi(b_i)$ and use Lemma \ref{lemmino} in order to find $t \in R \setminus \I $ such that $r_it \in \I$ for every $i$.  We claim that $atf \in I$ for every $f \in \G$ and, since each $ft \not \in D$, this will provide an infinite set of linearly independent elements of $I$. 
Write $atf= a_ib_itf$ and prove $ b_itf \in D $. For this, following the previous notation, write $$ b_itf =(r_i + s_1g_1 + \ldots + s_lg_l)tf $$ where $s_j \in \I$ and $g_j$ are products of elements of $\h$. Clearly $r_itf \in D$ since $r_it \in \I$. For the other terms, if $g_j \not \in \h$, it is possible to write $s_jg_j = s^{\prime} g^{\prime} s g$ where $s, s^{\prime} \in \I$,  $g^{\prime}$ is a product of elements of $\h$ such that $ s^{\prime} g^{\prime} \in D $, and $g \in \h$.
Thus $s_jg_jtf= s^{\prime} g^{\prime} ts gf \in D $, since $ts \in I$ and $ gf \in \h \G \subseteq \F $.
This proves the claim. 

Now, if $I$ was finitely generated by some elements $(d_1, \ldots, d_c)$, it would follow that the set $ W= \bigcup_{i=1}^n \h(d_i) $ is finite. But instead, it is possible to take $f \in \G \setminus W $ such that $atf \in I \setminus I^2$ and this is a contradiction. It follows that $I$ is necessarily not finitely generated.
\end{proof}

\begin{remark}
The theorem above can be used to construct many examples of non-integrally closed SBID. Indeed, taking an $\m$-primary ideal $\I \subsetneq \m$ and $x \in \m \setminus \I$, if $x^e \in \I$, it is sufficient to have the existence of $f \in \F$ such that also $f^e \in \F$ in order to have $xf$ integral over $ R(\I, \F) $ but not inside it.
\end{remark}

Despite the fact that many examples are not integrally closed, one may use this construction in order to find examples of completely integrally closed SBID. A known example of completely integrally closed Schreier SBID is the Shannon extension appearing in \cite[Corollary 7.7]{HOT}. Here we want to indicate how to construct many more non-Schreier examples. In the following theorem we assume $\I$ to be equal to the maximal ideal $\m$ of $R$. 

\begin{definition} \label{full}
Let $\F$ be an almost multiplicatively closed transcendent set over $R$. We say that $\F$ is \it full \rm if, given some elements $f_1, f_2, \ldots, f_s \in \F$ such that the product $ f_1 f_2 \cdots f_s \in \F $, then also $ f_{i_1} f_{i_2} \cdots f_{i_r} \in \F $ for every choice of $i_1, i_2, \ldots, i_r \in \lbrace 1,2, \ldots, s \rbrace$.
\end{definition}

\begin{theorem} \label{cic}
Let $D=R(\m,\F)$ defined as above in this section and consider its overring $T=R[\F]$. Assume the following conditions:
\begin{enumerate}
\item[(i)] $T$ is completely integrally closed.
\item[(ii)] $\F$ is full.
\item[(iii)] For every $f \in \F$, there exists $N \in \N$ such that $f^n \not \in \F$ for $n \geq N.$ 
\end{enumerate}
Then $D$ is completely integrally closed.
\end{theorem}

\begin{proof} Since $T$ is completely integrally closed, any almost integral element over $D$ has to be in $T$. Let $x \in T \setminus D $ and take $d \in D$, we want to consider the products $dx^n$ for $n \geq 1$ and prove that they are not in $D$ for any large enough $n$. The product of two elements $x,y \in T$ belongs to $D$ if and only if all the products of the monomials of $x$ and $y$ are in $D$. Hence, we may reduce to the case where $d$ and $x$ are monomials. Since $T$ is completely integrally closed, we must have $\bigcap_{i=1}^{\infty} \m = (0).$

Write $d= r f_1 \ldots f_l$ with $r \in R$, $f_i \in \F$ and $l \leq k$ where $k$ is the largest power such that $r \in \m^{k}$, and write $x= s g_1 \ldots g_m$ with $s \in R$, $g_i \in \F$ and $m > j$ where $j$ is the largest power such that $s \in \m^{j}$. Since we consider generic products of the form $ dx^n $, and $x=wg_{j+1} \ldots g_m$ with $w \in D$, we may also assume without loss of generality $s = 1$. Moreover, since $\F$ is full, we may assume all the products of the $f_i$ and all the products of $g_j$ to be not in $\F.$ 

The condition $ dx^n = r f_1 \ldots f_l (g_1 \ldots g_m)^n \in D$ implies that either $ l+nm \leq k $ or $f_ig_j^n \in \F$ for some $i,j$. Taking $n$ large enough, it is possible to exclude the first possibility and assume the second. Since $\F$ is full, this implies $g_j^n \in \F$. Assumption (iii) excludes this to be possible for every $n$. This implies that $D$ is completely integrally closed.
\end{proof}


\begin{example} 
We describe two examples of completely integrally closed SBID obtained using Theorem \ref{cic}.
\begin{enumerate}
\item  Set as base ring the DVR $R=k[a]_{(a)}$ and choose $\I= \m = aR.$ Consider, as in Example \ref{exc}, the set $ \F = \lbrace  x_{i_1}x_{i_2} \cdots x_{i_r}, \mbox{ squarefree monomial} \rbrace, $ where $ \lbrace x_n \rbrace_{n \geq 1} $ is a countable set of indeterminates over $K=k(a)$. 

In this case the ring $R[\F]=k[a, x_1, x_2, x_3, \ldots ]$ is completely integrally closed. Moreover $\F$ fulfills condition (b) and (c) of Theorem \ref{cic}. Indeed, a product of squarefree monomials is in $\F$ only if it is still squarefree and this happens if and only if the variables dividing the monomials are all distinct. Moreover, any divisor of a squarefree monomial is still squarefree and this makes $\F$ full. Condition (iii) is easily verified by taking $N=2$ for every element of $\F$. Thus, Theorem \ref{cic} implies that $R(\m,\F)$ is completely integrally closed.
\item The ring of the preceding example has infinite Krull dimension. This second example shows that it is possible to obtain completely integrally closed examples also of finite dimension. Again take $R=k[a]_{(a)}$ and $\I= \m = aR.$ Set $S = \mathbb{Q} \cap (0,1)$, let $x$ be an indeterminate over $K=k(a)$, and define $$  \F= \lbrace x^s \mbox{ : } s \in S \rbrace. $$ The ring $R[\F] = R[x^s \mbox{ : } s \in \mathbb{Q}_{+}]$ is completely integrally closed. The set $\F$ is full since, given $s_1, s_2, \ldots, s_l \in S$ such that $ s_1 +s_2+ \ldots+s_l \in S $, then clearly also any sum of the form $s_{i_1} + \ldots + s_{i_r} \in S$ if $i_1, \ldots, i_r \in \lbrace 1,2, \ldots, l \rbrace$. Finally, if $s= a/b$, then $bs \geq 1$, implying $bs \not \in S$ and hence $(x^s)^e \not \in \F$ for every $e \geq b$. Thus condition (i)-(ii)-(iii) of Theorem \ref{cic} are satisfied and $R(\m,\F)$ is completely integrally closed. One can easily observe that $R(\m,\F)$ has Krull dimension 1.
\end{enumerate}
\end{example}

\medskip

We switch now to consider the case in which the ideal $\I$ is not $\m$-primary. Clearly in this case $R(\I, \F)$ will often  not be a Strong Bezout Intersection domain, but this may still happen in some case. First we show that $R(\I, \F)$ fails to be a SBID in the case in which $R$ is not $t$-local and $\I$ is contained in some maximal $t$-ideal. 

\begin{prop} \label{tmaxc}
Let $R$ be a local domain with maximal ideal $\m$ and assume let $\I \subseteq \pp \subsetneq \m$ for some maximal $t$-ideal $\pp$ of $R$. Then
$ D= R(\I,\F) $ is not a SBID.
\end{prop}

\begin{proof} 
By Theorem \ref{carattsbid}, is sufficient to prove that $D$ is not $t$-local. Since $R$ is not $t$-local, there exists $J = (r_1, r_2, \ldots, r_n) \subseteq R $ such that $J_v=R$. We may assume without loss of generality that $r_n \not \in \pp$. Hence, $ r_1, r_2, \ldots, r_n \subseteq \gamma R $ for some $ \gamma \in K $ if and only if $R \subseteq \gamma R$. Assume $ r_1, r_2, \ldots, r_n \subseteq \gamma D $ for some $ \gamma \in Q(D) \setminus K $. Write $\gamma=\frac{h_1}{h_2}$ with $h_1,h_2 \in D$ and assume by way of contradiction $\gamma^{-1} \not \in D$. 

First assume $h_1=r \in \m$. Now, not all the coefficients in $R$ of the monomials of $h_2$ can be divisible by $r$, otherwise, again dividing common factors, we would have $\gamma^{-1} \in D$. But, $h_2J \subseteq rD$, and if $a \in R$ is a coefficient of $h_2$ not divisible by $r$, we are forced to have $aJ \subseteq rR$ and this is impossible since $J_v=R$ and $(\frac{r}{a})^{-1} = \frac{a}{r} \not \in R$.

Therefore, we may assume there exists some $s \in D \setminus R$ such that $s$ divides $h_1$ but no factor of $s$ divides $h_2$. It follows that $s$ divides all the $r_i$ inside $D$ and in particular divides $r_n$. Hence $r_n=sd$ with $d \in D$ and this is impossible because it would imply $r_n \in \mathcal{I} \subseteq \pp.$
\end{proof}

We deal now with the cases in which $R$ has a divided prime ideal. We recall that a prime ideal $\pp$ of a ring $R$ is \it divided \rm if $\pp \subseteq (x)$ for every $x \not \in \pp$ (the maximal ideal of a local domain is always divided).

\begin{prop} \label{prime}
Let $R$ be a local domain with maximal ideal $\m$ and let $\pp \subseteq \m$ be a divided prime ideal of $R$. Then 
 $ D= R(\pp,\F) $ is a SBID if and only if $\frac{R}{\pp}$ is a SBID.
\end{prop}

\begin{proof} 
Take the usual notation of this section. Let $d =r_0 + r_1g_1 + \ldots + r_lg_l \in D$ and assume $\pi(d)=r_0 \not \in \pp$. Since for $i \geq 1$, $r_i \in \pp$ and $\pp$ is divided, we get $\frac{r_i}{r_0} \in \pp$ and, if $r_i \in \pp^k$, also $\frac{r_i}{r_0} \in \pp^k$. Hence $\frac{r_i}{r_0} g_i \in \mathfrak{M} \subseteq D$ and $$d = r_0 \left( 1 + \frac{r_1}{r_0} g_1 + \ldots + \frac{r_i}{r_0} g_l \right),$$ implying $(d)=(r_0)$. 
Let $Q=\pp + \pp \F$ and take $q= p_0 + p_1h_1 + \ldots + p_mh_m \in Q$ and $r \in \mathfrak{M} \setminus Q$. By what said above, up to multiply a unit of $D$, we may assume $r \in R$. Replying a similar argument as before, we get that $ \frac{p_i}{r} \in \pp $ for every $i$, and hence $ \frac{q}{r} \in Q $. It follows that $Q$ is a divided prime ideal of $D$ and thus $D$ occurs in the pullback diagram \begin{center}
\begin{tikzcd}
   D \arrow[rightarrow]{r}\arrow[hookrightarrow]{d} 
  &  \frac{D}{Q} \arrow[hookrightarrow]{d}
  \\ 
   D_Q \arrow{r}
 &  \kappa(Q)
\end{tikzcd}
\end{center}
We conclude the proof by applying Theorem \ref{pullback}.
\end{proof}

Notice that the last proof does not depend on all the properties of the set $\F$ given in Definition \ref{almostmultiplicative},  hence it may be applied also in more general contexts. Also the first case described in Theorem \ref{costructionloccic} and the proof of Theorem \ref{kplusm} will not depend on the properties of the set $\F$.

We show now that, starting with rings that are already Strong Bezout Intersection such as rings with locally cyclic maximal ideal or $k+\m$ constructions, the construction $ R(\I,\F) $ will produce a Strong Bezout Intersection domain for any ideal $\I$.

\begin{theorem} \label{costructionloccic}
Let $R$ be a local domain with locally cyclic maximal ideal $\m$ and let $\I \subseteq \m$ be any ideal of $R$. Then the ring $ D= R(\I,\F) $ is a Strong Bezout Intersection domain.
\end{theorem}

\begin{proof} We separate the proof in two different cases: \\
\bf Case 1: \rm $\I=\I\m$ \\
In this case let $d_1, d_2, \ldots, d_n \in \mathfrak{M}$ be a finite number of non-units of $D$. Write each $d_j$ as usual as $d_j =r_{0_j} + r_{1_j}g_{1_j} + \ldots + r_{l_j}g_{l_j}$ where $g_{i_j}$ are products of elements of $\F$, $r_{i_j} \in \I$ if $i \geq 1$, and $ r_{0_j} \in \m $. 

Observe that, if $z \in x\I$ and $x \in yR$, then also $z \in y\I$.
Hence, using the facts that $\m$ is locally cyclic and $\I=\I\m$, it is possible to find $x \in \m$ such that, for every $j$, $ r_{0_j} \in xR $ and, for every $j$ and for $i \geq 1$, $r_{i_j} \in x\I$. It follows that $ d_1, d_2, \ldots, d_n \in xD $ and hence the maximal ideal of $D$ is locally cyclic. Corollary \ref{corsbid} implies that $D$ is a SBID. \\
\bf Case 2: \rm $\I\m \subsetneq \I$ \\
In this case we proceed along the line of the argument used to prove Theorem \ref{primary}. Consider some pairwise incomparable elements $a_1, a_2, \ldots, a_n \in D$ and let all the elements $ a, b_1, b_2, \ldots, b_n$ and the set $\G$ be defined in the same way as in that proof. Also set $r_i=\pi(b_i)$. To conclude the proof using the same argument we only need to find $t \in R \setminus \I$ such that $r_it \in \I$ for every $i$. 

In light of Theorem \ref{primary}, we may assume the radical of $\I$ to be properly contained in $\m$.
If $r_i \in \I$ for every $i$, simply set $t=1$. Otherwise, choose $z \in \I \setminus \I \m$. Use the fact that $\m$ is locally cyclic to find $x \in \m \setminus \I$ such that $r_1, r_2, \ldots, r_n, z \in xR$ and set $t=\frac{z}{x}$.
\end{proof}

\begin{theorem} \label{kplusm}
Let $T$ be a local domain with maximal ideal $\m$ and containing its residue field $k(X)$. Let $R=k+\m$ let $\I \subseteq \m$ be any ideal of $R$. Then the ring $ D= R(\I,\F) $ is a Strong Bezout Intersection domain.
\end{theorem}

\begin{proof} Using the facts that $T$ and $R$ has the same quotient field and $\I$ is also an ideal of $T$, it is possible to define the local domain $T(I,\F)$. Let $\mathfrak{n}= \m+ \I\F$ be the maximal ideal of $T(I,\F)$. 
Clearly $\mathfrak{n}$ is also the maximal ideal of $D$ and $D= k + \m + \I\F$. Hence $D = k + \mathfrak{n}$ is a $k + \m$ construction and it is a SBID by Theorem \ref{k+m}. 
\end{proof}

\section*{Acknowledgements}
The first author has received financial support from Indam (Istituto Nazionale di Alta Matematica) to spend three months of Spring 2019 at Ohio State University. This article has been carried out during this period of time and the author acknowledges the support of Indam for giving him this opportunity of research. 

The first author is also supported by the NAWA Foundation grant Powroty "Applications of Lie algebras to Commutative Algebra".

 

\end{document}